\documentclass[12pt,fleqn]{amsart}

\usepackage[all]{xy}
\usepackage{tikz-cd}
\usepackage{amscd}
\usepackage{setspace}
\usepackage{korpi-lib}
\usepackage{color}
\usepackage{enumerate}
\usepackage{amssymb}
\usepackage{hyperref}

\usepackage{comment}

\newcommand{\bmd}{d_{\rm BM}}
\newcommand{\bmdp}{d_{\rm BM}^+}

\title{On positive  Banach--Mazur distance}
\author[M. Korpalski]{Maciej Korpalski}
\author[G.\ Plebanek]{Grzegorz Plebanek}
\address{Instytut Matematyczny, Uniwersytet Wroc\l awski, pl.\ Grunwaldzki 2, 50-384 Wroc\-\l aw\\ Poland}
\email{Maciej.Korpalski@math.uni.wroc.pl \\ Grzegorz.Plebanek@math.uni.wroc.pl}
\date{}

\begin{document}

\begin{abstract}
In this working note we study the one-sided positive Banach--Mazur distance in order to address questions posed in  
\cite{CHRS26}.
Building on  methods developed in \cite{KP25}, we prove that
\[d_{\rm BM}^+\big( C([0,\omega] \times 2), C[0,\omega]\big)\ge \frac{3+\sqrt{17}}{2}>3.56\]
and
\[d_{\rm BM}^+\big(  C[0,\omega], C([0,\omega] \times 2)\big) >3 .\]
In particular, both quantities exceed $3 = d_{\rm BM}\big( C([0,\omega] \times 2), C[0,\omega]\big)$, which solves one of the problems.
\end{abstract}

\maketitle

\section{Introduction}
Given two isomorphic Banach spaces $X$ and $Y$, their Banach--Mazur distance $\bmd(X,Y)$ is defined as the infimum of the distortions $\|T\|\cdot\|T^{-1}\|$ taken over all isomorphisms $T:X\to Y$.
If $K$ and $L$ are compact Hausdorff spaces and $C(K),C(L)$ denote the Banach spaces of real-valued continuous functions with the usual supremum norm, then we can consider, as in \cite{CHRS26}, the positive Banach--Mazur distance $\bmdp\big( C(K), C(L)\big)$ which is defined as the infimum of $\|T\|\cdot\|T^{-1}\|$ taken over all {\bf positive} isomorphisms $T: C(K) \to C(L)$, so isomorphisms satisfying $Tf \ge 0$ whenever $f \ge 0$.
Note that this distance is one-sided --- in general, $\bmdp\big( C(K), C(L)\big) \neq \bmdp\big( C(L), C(K)\big)$.
We study $\bmdp\big( C(K), C[0,\omega]\big)$ and $\bmdp\big( C[0,\omega], C(K)\big)$, where $K=[0,\omega] \times k$ for a natural number $k$ (in the current version of the preprint equal to $2$). 

Our main goal in this preprint is to answer the following questions.

\begin{question}{\cite{CHRS26}*{Question 0.2.}} \label{q:1}
What are the values of $\bmdp\big( C[0,\omega], C([0,\omega] \times 2)\big)$ and $\bmdp\big( C([0,\omega] \times 2), C[0,\omega]\big)$?
\end{question}

\begin{question}{\cite{CHRS26}*{Question 0.3.}}\label{q:2}
Are there countable ordinals $\omega \le \alpha < \beta$ such that 
\[\bmd\big( C[0,\alpha], C[0,\beta]\big) <
\min\{\bmdp\big( C[0,\alpha], C[0,\beta]\big), \bmdp\big( C[0,\beta], C[0,\alpha]\big)\}?\]
\end{question}

 In Sections \ref{sec:1->2} and \ref{sec:2->1}, we show that 
 \[\bmdp\big( C[0,\omega], C([0,\omega] \times 2)\big) > 3\mbox{ and }
 \bmdp\big( C([0,\omega] \times 2), C[0,\omega]\big)>3.\]
Since it is well-known, see Gordon  \cite{Go70},
 that $\bmd\big( C[0,\omega], C([0,\omega] \times 2)\big) = 3$, 
this answers Question \ref{q:2} and makes a  progress on Question \ref{q:1}.

\section{Preliminaries}
 For two compact spaces $K, L$, we call an isomorphism $T : C(K) \to C(L)$ norm-increasing if $\|Tf\| \ge \|f\|$ for all $f \in C(K)$.
We will consider only norm-increasing isomorphisms, since any isomorphism can be rescaled to have this property.
For each $y\in L$, we write $\nu_y$ for the measure on $K$ defined for $g\in C(K)$ by $\nu_y(g)=Tg(y)$.
Note that the measures $\nu_y$, for $y\in L$, form a 1-norming subset of $M(K)$.
If additionally $T$ is positive, then every measure $\nu_y$ is positive.

In the case $L=[0,\omega]$, we write $\nu_i$ for $i\in [0,\omega]$, and denote $\nu_\omega$ simply by $\nu$. 
Note that $\nu_i\to\nu$ in the $weak^\ast$ topology on $M(K)$.

We mostly consider the case $K = [0,\omega] \times 2$, $L=[0,\omega]$ and $T : C(K) \to C(L)$. 
To estimate $\bmdp\big( C([0,\omega] \times 2), C[0,\omega]\big)$, we will assume that $T$ is positive. 
Similarly, for estimating $\bmdp\big( C[0,\omega], C([0,\omega] \times 2)\big)$, we will assume that $T^{-1}$ is positive.

Let us fix notation for the remainder of the paper.
Consider a norm-increasing isomorphism $T: C([0,\omega] \times 2) \to C[0,\omega]$ and denote $t = \|T\|$.

Given  $m \in \{0, 1\}$, 
denote the singleton of the endpoint as $d_m=\{(\omega, m)\}$. 
Write $\nu(d_m)=\theta_m$, so that
\[ \nu=\nu'+\theta_0\cdot \delta_{(\omega,0)}+\theta_1\cdot \delta_{(\omega,1)},\]
where $\nu'$ vanishes on both endpoints. Put $\theta' = \|\nu'\|$.

We will commonly use the asymptotic symbol $\lesssim$ in the following sense: 
$a\lesssim b$ means that the real-valued functions $a$ and $b$ defined for $\eps>0$ satisfy $\lim_{\eps\to 0^+} a(\eps)\le \lim_{\eps\to 0^+} b(\eps)$.

We write  $A_m(n) = [n, \omega] \times \{m\}$ for $n \in \omega$, and denote by 
 $\chi_A$  the characteristic function of a set $A$.
Further,  $e_x$ stands for the basic function equal to $1$ at $x$ and $0$ elsewhere.

If  we  choose $n_0$ large enough then $|\nu|(A_m(n_0)) \approx \theta_m$, and this means 
that  the measure $\nu'$ is essentially supported by the complement of $A_0(n_0)\cup A_1(n_0)$.
We define two parameters:
\[ s_m=\max\big(\sup_i |\nu_i(d_m)|, |\theta_m| \big),\quad  m = 0, 1.\]

\section{On \texorpdfstring{$\bmdp\big( C([0,\omega] \times 2), C[0,\omega]\big)$}{dBM+(C([0, ω]x2), C[0, ω])}} \label{sec:2->1}

In this section, we consider a norm-increasing positive isomorphism 
\[ T:C([0,\omega] \times 2) \to C[0,\omega],\]
 and estimate  $t = \|T\|$ from below. 
 
 Recall  that in such a  case the measures $\nu_i$ are nonnegative.
This fact is essential for the  following two lemmas.

\begin{lemma}\label{2:1} 
We have
\[ \theta_1\ge 1+\frac{\theta_0}{s_0}\mbox{ and } \theta_0\ge 1+\frac{\theta_1}{s_1}.\]
\end{lemma}

\begin{proof}  
Fix $\eps > 0$.
To check the first inequality, consider the function
\[ f_n=e_{(n,1)}- \frac{1}{s_0+\eps}\chi_{A_0(n_0)}.\]

Since $\nu_i \to \nu$ in the $weak^\ast$ topology, there is $i_0$ such that $\nu_i (A_0(n_0)) \approx \theta_0$ for all $i\ge i_0$. 
If $n$ is sufficiently large, then $\nu_i(e_{(n,1)}) \approx 0$ for all $i < i_0$. 
It follows that no measure $\nu_i$ with $i < i_0$ can norm the function $f_n$. 
Hence there is $i(n)\ge i_0$ such that $\nu_{i(n)}(f_n) \ge 1$, and thus
\[ \nu_{i(n)}(e_{(n,1)}) - \frac{\theta_0}{s_0+\eps} \gtrsim 1.\] 
Note that $i(n) \to \infty$, so 
\[ \nu_{i(n)}(e_{(n,1)})\le \nu_{i(n)}(A_1(n_0))\approx \theta_1,\]
and hence
\[ \theta_1-\frac{\theta_0}{s_0+\eps}\gtrsim 1.\] 
This yields  the first inequality and the argument for the second one is symmetric.
\end{proof}

The following lemma is a variation of \cite{KP25}*{Lemma 5.6}.

\begin{lemma}\label{2:2} 
If $s_0>\theta_0, s_1>\theta_1$, then 
\[ \|\nu'\|\ge 1-\Big(\theta_0\frac{t-s_0-1}{s_0}+ \theta_1\frac{t-s_1-1}{s_1}\Big).\]
\end{lemma}

\begin{proof}
Fix $\eps>0$.
Without loss of generality, assume that
\[ s_0=\nu_{i_0}(d_0)\text{ and } s_1=\nu_{i_1}(d_1),\]
where $i_0, i_1$ are not necessarily distinct.

We pick a function $g\in C\big([0,\omega] \times 2\big)$ such that $\nu_{i_m}(g)=-1$ for $m=0,1$ and $\nu_i(g)=1$ for every $i\neq i_0, i_1$. 
Since $T$ is norm-increasing, we have $\|g\|\le 1$; denote $g_m=g(\omega, m)$.

For each $m = 0, 1$, let $\mu_{i_m}$ denote the part of $\nu_{i_m}$ living outside $d_m$. Then $\|\mu_{i_0}\|\le t-s_0$ and $\|\mu_{i_1}\|\le t-s_1$, so 
\[ -1=\nu_{i_0}(g)=\mu_{i_0}(g)+g_0\nu_{i_0}(d_0)\ge g_0\nu_{i_0}(d_0)-(t-s_0),\]
\[ -1=\nu_{i_1}(g)=\mu_{i_1}(g)+g_1\nu_{i_1}(d_1)\ge g_1\nu_{i_1}(d_1)-(t-s_1),\]
thus
\[ g_0\le \frac{t-s_0-1}{s_0} \text{ and } g_1\le \frac{t-s_1-1}{s_1}.\]
On the other hand,
\[ 1=\nu(g)=\nu'(g)+g_0\theta_0+g_1\theta_1,\]
\[ \|\nu'\|\ge \nu'(g)=1-\big(g_0\theta_0+g_1\theta_1\big)\ge 1-\Big(\theta_0\frac{t-s_0-1}{s_0}+\theta_1\frac{t-s_1-1}{s_1}\Big),\]
and we are done.
\end{proof}

\begin{theorem}\label{2:3}
\[\bmdp\big( C([0,\omega] \times 2), C[0,\omega]\big)\ge \frac{3+\sqrt{17}}{2}>3.56.\]
\end{theorem}

\begin{proof}
If $s_0=\theta_0$ and $s_1=\theta_1$, then it follows from Lemma \ref{2:1} that $\theta_0,\theta_1\ge 2$, and hence $t\ge 4$.

In the mixed case, suppose that $s_0=\theta_0$. Then Lemma \ref{2:1} gives 
\[ \theta_1\ge 2, \theta_0\ge 1+\frac{2}{t}, \mbox{ so}\]
\[t\ge\theta_0+\theta_1\ge 3+\frac{2}{t}.\]
Then $t^2 - 3t - 2 \ge 0$, which gives $t\ge (3+\sqrt{17})/2$.

It remains to consider the case where $s_0>\theta_0$, $s_1>\theta_1$. Then Lemmas \ref{2:1} and \ref{2:2} give
\numberwithin{equation}{section}

\begin{equation}\label{e1}
\theta_1\ge 1+\frac{\theta_0}{s_0}
\end{equation}

\begin{equation}\label{e2}
\theta_0\ge 1+\frac{\theta_1}{s_1}
\end{equation}

\begin{equation}\label{e3}
\|\nu'\|\ge 1-\Big(\theta_0\frac{t-s_0-1}{s_0}+ \theta_1\frac{t-s_1-1}{s_1}\Big)
\end{equation}

Equation \eqref{e3} is equivalent to 
\[ \|\nu'\|\ge 1+\theta_0+\theta_1-(t - 1)\Big(\frac{\theta_0}{s_0}+ \frac{\theta_1}{s_1}\Big),\]
so, using \eqref{e1} and \eqref{e2},
\[ \|\nu'\|\ge 1+\theta_0+\theta_1-(t - 1)\big(\theta_0+\theta_1-2\big)=2t - 1-\big(\theta_0+\theta_1\big)(t-2),\]

Since $t\ge \|\nu\|=\|\nu'\|+\theta_0+\theta_1$, it follows that
\[t\ge 2t - 1- \big(\theta_0+\theta_1\big)(t-3), \mbox{ that is}\]
\[ \theta_0+\theta_1\ge \frac{t - 1}{t-3}.\]

Finally,
\[t\ge \theta_0+\theta_1\ge \frac{t - 1}{t-3},\]
so $t^2-4t+1\ge 0$, which implies 
\[ t\ge 2+\sqrt{3}\quad \Big(>\frac{3+\sqrt{17}}{2}\Big).\]
Thus the smallest lower bound arises from the mixed case.
\end{proof}

\section{On \texorpdfstring{$\bmdp\big( C[0,\omega], C([0,\omega] \times 2)\big)$}{dBM+(C[0, ω], C([0, ω]x2))}}\label{sec:1->2}

We now consider a norm-increasing isomorphism $T: C([0,\omega] \times 2) \to C[0,\omega]$ such that $T^{-1}$ is positive. 
In other words, we assume that $Tf\ge 0$ implies $f\ge 0$.
We first prove   the following general fact, 
which will be especially useful in the case where either $\theta_0$ or $\theta_1$ is negative.

\begin{lemma}\label{3:0}
Suppose that $h\in C([0,\omega] \times 2)$ and $h(n,1)=1$ for some $n$.

Then there is $i$ such that $\nu_i(h)\ge 1$.
\end{lemma}

\begin{proof}
There is $\wh{h}\in C([0,\omega] \times 2)$ such that, for every $i$, $\nu_i(\wh{h})=\nu_i(h)$ whenever $\nu_i(h)\ge 0$ and $\nu_i(\wh{h})=0$ otherwise. 

Then $\nu_i(\wh{h}-h)\ge 0$ for every $i$. Since $T^{-1}\ge 0$, we have 
\[\wh{h}-h\ge 0, \text{ so } \|\wh{h}\|\ge \wh{h}(n,1)\ge h(n,1)=1.\]

Thus there is $i$ satisfying $|\nu_i(\wh{h})|\ge 1$ which means $\nu_i(h)\ge 1$.
\end{proof}

The further analysis  depends  on the signs of $\theta_0$ and $\theta_1$, so we consider two cases separately.

\subsection{Opposite signs} \label{ssec:op}
Consider the case $\theta_0 \ge 0 \ge \theta_1$.
In this case,  Lemma \ref{3:0} yields the following.

\begin{proposition}\label{3:2}
Suppose that $\theta_0 \ge 0 \ge \theta_1$. Then
\[ t\ge 2\Big( 1+\frac{\|\nu\|}{t} \Big)+\|\nu\|.\]
\end{proposition}

\begin{proof}
We have $\|\nu\|=\theta_0 - \theta_1 + \theta'$ (where $\theta'=\|\nu'\|$).
Fix $\eps > 0$;
there is a continuous function $\vf$ on $[0,\omega]\times 2$
 such that $\|\vf\| \le 1$ and  $\nu(\vf) >\|\nu\| - 2\eps$.
Then there is  $i_0$ such that  $\nu_i(\vf) > \|\nu\| - \eps$ for all $i > i_0$.

Since $\theta_1 \le 0$, we can assume that $\lim_n \vf(n,1)\le 0$,
Fix some $n\ge n_0$ such that $\vf (n, 1) <\delta $ and $|\nu_i(n, 1)|<\delta $  for all $i \le i_0$;
$\delta>0$ will be fixed in a while.
We apply Lemma~\ref{3:0} to
\[ h=e_{(n,1)} - \frac{1}{t+\eps}\vf .\]

Then for $i \le i_0$ we have
\[\nu_i(h) = \nu_i(e_{(n,1)}) - \frac{1}{t+\eps}\nu_i(\vf) \le \delta + \frac{t}{t+\eps} < 1.\]
On the other hand, $\|h\|>1-\delta$, so,  by taking $\delta$ sufficiently small,
we can guarantee that $\nu_i(h)>1-\delta$ must hold for some $i\ge i_0$. 

Then 
\[ 1-\delta  \le \nu_i(h) = \nu_i(e_{(n,1)}) - \frac{1}{t+\eps}\nu_i(\vf)<  \nu_i(e_{(n,1)} ) - 
\frac{1}{t+1} (\|\nu\|-\eps).\]

To state it briefly: 
\[ \nu_i(e_{(n,1)})\gtrsim 1+\frac{\|\nu\|}{t}.\]

Since $n_0$ was chosen so that 
$\nu_i(A_1(n_0))\approx \theta_1\le  0$, we estimate $|\nu_i|(A_1(n_0))$ as
\[ |\nu_i|(A_1(n_0)) \gtrsim \nu_i(2e_{(n,1)} - \chi_{A_1(n_0)}) \gtrsim  2 \Big(1 + \frac{\|\nu\|}{t}\Big) - \theta_1,\]
and then $\|\nu_i\|$ as
\[ \|\nu_i\| \gtrsim  |\nu_i|(A_1(n_0)) + \theta_0+\theta'\gtrsim 2 \Big(1 + \frac{\|\nu\|}{t}\Big)+\|\nu\|.\]
\end{proof}

\begin{corollary}\label{3:3}
If $\theta_0 \ge 0 \ge \theta_1$, then $t \ge (3+\sqrt{17})/2$.
\end{corollary}

\begin{proof}
Note that $\|\nu\|\ge 1$ so
Lemma \ref{3:2} gives
\[ t\ge 2\Big(1 + \frac{1}{t} \Big) +1= 3 + \frac{2}{t}. \]
Thus $t^2 - 3t - 2 \ge 0$ and $t \ge (3+\sqrt{17})/2$.
\end{proof}

\subsection{The same signs} \label{ssec:same}
We consider the case where $\theta_0\cdot\theta_1> 0$ but analyze
it under the additional assumption that $\nu'=0$.

Without loss of generality, assume that $|\theta_0|\ge |\theta_1|$, and consider  functions of the form
\[ f_n=\chi_{A_1(n)}-\theta_1/\theta_0\chi_{A_0(n)}.\]
For any $n$ we have  $\nu(f_n)= 0$ and, consequently,  $\nu_i(f_n)\approx 0$ holds for large $i$.

\begin{lemma}\label{3:4}
If $\theta_0\cdot \theta_1> 0$ then there is $k$ such that $|\nu_i(f_k)|  \lessapprox t - 1$ for {\bf all}  $i$.
\end{lemma}

\begin{proof}
Take any  $n$ and $i_0$ such that $|\nu_i(f_n)|\approx 0 $ for $i\ge i_0$.
Then take  $g\in C([0,\omega] \times 2)$ such that $\nu_i(g)=1$ for  $i<i_0$ and $\nu_i(g)=0$ whenever $i\ge  i_0$.
We have $g\ge 0$ since $T^{-1}$ is a positive operator. Write $g_m=g(\omega,m)$.

It follows that  $\nu(g)=g_0 \theta_0 + g_1 \theta_1  = 0$;  hence  $g_0= g_1=0$ (here we use
$\theta_0\cdot \theta_1>0$).
Consequently, the variation of $\nu_i$ outside the endpoints is at least $1$ for every $i<i_0$.

Therefore there is $k>n$ such that $|\nu_i(f_k)|\lesssim t-1$ for all $i<i_0$.
Now it suffices to notice that for $i\ge i_0$ we have $|\nu_i(f_k)|\lesssim t/2<t-1$;
see \cite[Lemma 2.3(b)]{KP25} if necessary.
\end{proof}

\section{Conclusions}

We have proven that
\[\bmdp\big( C([0,\omega] \times 2), C[0,\omega]\big)\ge \frac{3+\sqrt{17}}{2}>3.56.\]
For the opposite direction, our analysis is incomplete but we get the following.

\begin{corollary}
\[ \bmdp\big( C[0,\omega], C([0,\omega] \times 2)\big) >3. \]
\end{corollary}

\begin{proof}
By the results from section \ref{sec:1->2}, we
only need to treat the case $\theta_0\cdot\theta_1>0$.

Our basic lemma from \cite{KP25}, applied to the 
function of the form $f_n=\chi_{A_1(n)}-\theta_1/\theta_0\chi_{A_0(n)}$, gives
\[ t\ge 2\frac{t}{t-1} + \theta'.\]
Note that $t=3$ would imply $\theta'=0$ but then we can refer to  
Lemma \ref{3:4} and get
\[ t\ge 2\frac{t-1}{t-2},\]
since we can use the basic lemma from \cite{KP25}
with the parameter $t-1$ replacing $t$.
Then $t\ge 2+\sqrt{2}$.
\end{proof}

Surely, there is room for improvement: the best known upper bounds are $2 + \sqrt{5}$ and $2 + \sqrt{3}$ respectively.
However, both bounds are above $3 = \bmd\big( C([0,\omega] \times 2), C[0,\omega]\big)$, so we have answered Question \ref{q:2}.

\section{Remarks}

It seems possible to extend the results of Sections \ref{sec:2->1} and \ref{sec:1->2} to obtain bounds for the values $\bmdp\big( C([0,\omega] \times k), C[0,\omega]\big)$ and $\bmdp\big( C[0,\omega], C([0,\omega] \times k)\big)$ for $k > 2$.

It is also interesting that the bounds from \cite{CHRS26} for $\bmdp\big( C([0,\omega] \times k), C[0,\omega]\big)$ are the same as the bounds for $\bmd\big( C([0,\omega^k]), C[0,\omega]\big)$ from \cite{CG13}. 
This suggests that there might be a deeper connection between these distances.
It may, for example, be  possible to densely embed the space of positive isomorphisms from $C([0,\omega] \times k)$ to $C[0,\omega]$ into the space of all isomorphisms from $C([0,\omega^k])$ to $C[0,\omega]$.
Perhaps it is even possible to define some operation $\ol{}$ on compact spaces such that, for a compact space $K$, we have
\[\bmdp\big( C(\ol{K}), C[0,\omega]\big) = \bmd\big( C(K), C[0,\omega]\big).\]
This could shed some light on the structure of these isomorphisms and open up new methods for calculating Banach--Mazur distances.

\bibliography{refs}
\end{document}